\documentclass[a4paper,12pt]{article}
\usepackage[top=1.25in,left=1in,right=1in,bottom=1.25in]{geometry}
\usepackage{latexsym}
\usepackage{epsfig, ecltree, epic, eepic}
\usepackage{enumerate,amsmath,amssymb,dsfont,pifont,amsthm}
\usepackage[dvips]{color}
\usepackage{graphicx}
\usepackage[arrow, matrix, curve]{xy}

\theoremstyle{plain}
\newtheorem{theorem}{Theorem}[section]
\newtheorem{corollary}[theorem]{Corollary}
\newtheorem{lemma}[theorem]{Lemma}
\newtheorem{proposition}[theorem]{Proposition}

\theoremstyle{definition}
\newtheorem{definition}[theorem]{Definition}

\theoremstyle{remark}
\newtheorem{remark}[theorem]{Remark}
\newtheorem{remarks}[theorem]{Remarks}
\newtheorem{example}[theorem]{Example}

\newcommand{\bbc}{\mathbb{C}}
\newcommand{\bbr}{\mathbb{R}}

\newcommand{\bbp}{\mathbb{P}}
\newcommand{\bbe}{\mathbb{E}}
\newcommand{\bbf}{\mathbb{F}}
\newcommand{\bbn}{\mathbb{N}}

\newcommand{\cb}{\mathcal{B}}
\newcommand{\cf}{\mathcal{F}}

\newcommand{\pfe}{\longmapsto}
\newcommand{\pf}{\longrightarrow}

\newcommand{\abs}[1]{\left| #1 \right|}
\newcommand{\norm}[1]{\left\| #1 \right\|}

\newcommand{\gdw}{\Leftrightarrow}

\begin{document}

\allowdisplaybreaks

\title{\bfseries A Classification of Deterministic Hunt Processes with Some Applications}

\author{%
    \textsc{Alexander Schnurr}%
    \thanks{Lehrstuhl IV, Fakult\"at f\"ur Mathematik, Technische Universit\"at Dortmund,
              D-44227 Dortmund, Germany,
              \texttt{alexander.schnurr@math.tu-dortmund.de}}
    }

\date{}

\maketitle
\begin{abstract}
Deterministic processes form an important building block of several classes of processes. We provide a method to classify deterministic Hunt processes. Within this framework we characterize different subclasses (e.g. Feller) and construct some (counter-)examples. In particular the existence of a Hunt semimartingale (on $\bbr$) which is not an It\^o process in the sense of Cinlar, Jacod, Protter and Sharpe (1980) is proven.
\end{abstract}

\emph{MSC 2010:} 60J25 (primary), 60J35, 47G30 (secondary)

\emph{Keywords:} Hunt semimartingale, deterministic process, It\^o process, Feller semigroup, semimartingale characteristics, symbol

\section{Introduction}

Hunt semimartingales and It\^o processes have been studied extensively some 25 to 30 years ago. Nowadays they return into the focus of research. This is in particular due to the fact that practitioners working in the direction of mathematical finance have recognized that reality is more complex then suggested by the Brownian or OU-models. Some authors prefer L\'evy processes (cf. \cite{lp}, Part V), others It\^o semimartingales (see e.g. \cite{veraart10}) or Feller processes (see e.g. \cite{bar-lev01}) in order to model financial data. All these classes have in common that they do have a deterministic part which sometimes plays an important r\^ole. It is this part we are dealing with in the present paper. Deterministic Markov processes are treated only seldom in the literature. This is in particular due to the following fact:
\begin{proposition}
Every deterministic process is a \emph{simple} Markov process.
\end{proposition}
The word `simple' has to be emphasized here since this statement does not hold true for Markov families, which are sometimes called universal Markov processes (cf. \cite{bauerwt}, \cite{niels3}) or Markov processes in the sense of Blumenthal and Getoor (cf. \cite{blumenthalget}). Such a process $X=(X_t)_{t\geq 0}$ can start in every point of the respective state space and furthermore time homogeneity is present, i.e.
writing for $s,t\geq 0$, $x,z\in\bbr^d$ and a Borel set $B$ in $\bbr^d$ $P_{s,t}^x(z,B):=\bbp^x(X_{t}\in B | X_s=z)$ we have:
\begin{align} \label{timehomone}
P_{t-s}(z,B):=P_{s,t}^x(z,B)=P_{s+h,t+h}^y(z,B), \hspace{10mm} h\geq 0.
\end{align}
Since the `simple' processes are not of any interest in the given situation we will consider only families of processes. For the sake of readability we call the whole family $(X,\bbp^x)_{x\in\bbr^d}$ a \emph{stochastic process}. As it is always assumed that $\bbp^x(X_0=x)=1$ we write $X^x$ for the simple process $(X,\bbp^x)$. Since deterministic processes are adapted to every possible filtration, we will not write it down all the time but we assume that a fixed filtration $\bbf=(\cf_t)_{t\geq 0}$ as well as a $\sigma$-algebra $\cf$ on an arbitrary set $\Omega\neq \emptyset$ are always in the background.
The following example is the starting point of our considerations:
\begin{example} \label{ex:startingpoint}
Let $\Phi:\bbr\to\bbr$ be bijective, strictly monotonically increasing and such that $\Phi(0)=0$. In this case a Markov process is given by
\[
X_t^x(\omega):=\Phi(t+\Phi^{-1}(x)),\hspace{10mm} \text{ for every } \omega \in \Omega,
\]
i.e. by shifting the function $\Phi$ to the left and to the right. By inverting the function $x\mapsto \Phi(t+\Phi^{-1}(x))$ we know where a path being at time $t$ in $z\in\bbr$ has started at time zero:
\[
x=\Phi(\Phi^{-1}(z)-t)
\]
This gives us the homogeneous transition property \eqref{timehomone}: let $z,w\in\bbr$, $t,h\geq 0$ and $x\in\bbr$ such that $X_h^x=z$ then we obtain
\begin{align*}
P_{h,t+h}^x(z,\{w\}) = 1 &\gdw \Phi\Big((t+h)+\Phi^{-1}(x)\Big) = w \\
&\gdw \Phi\Big((t+h)+\Phi^{-1}(\Phi(\Phi^{-1}(z)-t))\Big) = w \\
&\gdw \Phi\Big(h+\Phi^{-1}(z) \Big)= w \\
&\gdw P_{0,t}^0(z,\{w\}) = 1.
\end{align*}
The function $\Phi$ will be called the `generating path' since it contains all the information of the process. Obviously the restriction $\Phi(0)=0$ is not needed and any shifted generating path $t\mapsto \Phi(t-s)$ ($s\in \bbr$) would have served as well.
\end{example}

In Section 2 we show that up to `dividing the state-space' every deterministic Hunt processes with state space $\bbr$ has the structure given above with a (countable) number of generating paths which define the process on disjoint intervals.

Since the definitions and notations for some of the classes of processes we are treating are not unified, let us first fix some terminology: a Markov process in the above sense, i.e. satisfying \eqref{timehomone}, is called \emph{Hunt process} if it is quasi-left continuous (cf. Definition I.2.25 of \cite{jacodshir}) with respect to every $\bbp^x$ $(x\in\bbr^d)$.
We restrict ourselves to Markov processes with right-continuous paths and associate a semigroup of operators with every such process: $(T_t)_{t\geq 0}$ on the bounded Borel measurable functions is given by
\[
T_t u(x) = \bbe^x u(X_t) = \int_\Omega u(X_t(\omega)) \, \bbp^x(d\omega)=\int_{\bbr^d} u(y) \, P_t(x,dy).
\]
We call $(T_t)_{t\geq 0}$ a \emph{Feller semigroup} and $(X_t)_{t\geq 0}$ a \emph{Feller process} if the following
conditions are satisfied: \\
\hspace*{5mm}$(F1)$ $T_t:C_\infty(\bbr^d) \to C_\infty(\bbr^d)$ for every $t\geq 0$, \\
\hspace*{5mm}$(F2)$ $\lim_{t\downarrow 0} \norm{T_tu-u}_\infty =0$ for every $u\in C_\infty(\bbr^d)$.\\
A Feller process is called \emph{rich} if the test functions $C_c^\infty(\bbr^d)$ are contained in the domain of its generator (cf. Definition \ref{def:generator}). Sometimes one encounters a different concept of Feller semigroups in the literature in which $C_\infty(\bbr^d)$ is replaced by the space $C_b(\bbr^d)$ equipped with local uniform convergence (cf. \cite{schilling98pos}). For the sake of clarity we will call such semi groups and the related processes \emph{$C_b$-Feller}.
We say that a process $(X,\bbp^x)_{x\in\bbr^d}$ is a \emph{semimartingale}, if every $X^x$ is one.
A Markov semimartingale is called \emph{It\^o process} (cf. \cite{vierleute}) if it has characteristics of the form:
\begin{align*}
B_t^{j}(\omega)&=\int_0^t \ell^{j}(X_s(\omega)) \ ds & j=1,...,d\\
C_t^{jk} (\omega)&=\int_0^t Q^{jk} (X_s(\omega)) \ ds &j,k=1,...,d\\
\nu(\omega;ds,dy)&=N(X_s(\omega),dy)\ ds
\end{align*}
where $\ell^{j},Q^{jk}:\bbr^d \pf \bbr$ are measurable functions, $Q(x)=(Q^{jk}(x))_{1\leq j,k \leq d}$ is a positive semidefinite matrix for every $x\in\bbr^d$, and $N(x,\cdot)$ is a Borel transition kernel on $\bbr^d \times \cb(\bbr^d \backslash \{0\})$.

In \cite{mydiss} it was shown that every rich Feller process is an It\^o process. Every Hunt semimartingale can be written as a random time change by results of \cite{cinlarjacod81}. The following diagram gives an overview on the interdependence of the classes of processes:
\begin{align*}
\begin{array}{ccccccccc}
   \text{L\'evy} & \subset  & \begin{array}{c}\text{(rich)} \\ \text{Feller} \end{array} & \subset & \text{It\^o}& \subset & \begin{array}{c}\text{Hunt} \\ \text{ semimartingale} \end{array}
         &\subset  &   \begin{array}{c}\text{Markov} \\ \text{ semimartingale} \end{array}   \\
                 & \rule[5mm]{0mm}{0mm} &    \cap & &&& \cap &&\cap\\
                 & \rule[5mm]{0mm}{0mm} & \text{Feller} & &\subset & &\text{Hunt}&\subset&\text{Markov}
\end{array}
\end{align*}

Let us give a brief outline on how the paper is organized: in Section 2 we will analyze deterministic Hunt processes. First we will deal with the behavior of single paths of the processes and afterwards with the dependence between these paths. These considerations lead to a result on the general structure of one-dimensional deterministic Hunt processes. As a byproduct we show that every such process is a semimartingale. Using our stuctural result we characterize the two kinds of Feller processes mentioned above and the property of being `rich'. Using these characterizations it is a comparably simple task to set up several examples and counterexamples in this context. In Section 4 we use the well known Cantor function to define a process which is a Hunt semimartingale and even a Feller process, but not an It\^o process. Further examples are considered in Section 5.

Most of the notation we are using is more or less standard. Note that we prefer to write $]s,t[$ for an open interval rather then $(s,t)$ and use the same convention for semi-open intervals. For the open ball of radius $r$ around $x\in\bbr^d$ we write $B_r(x)$. In the context of semimartingales we follow mainly \cite{jacodshir}.

\section{The Classification Theorem}

A stochastic process $(X,\bbp^x)_{x\in\bbr^d}=(X^x)_{x\in\bbr^d}$ is called \emph{deterministic} if it does not depend on $\omega$, i.e. there exists a function $f:\bbr^d \times [0,\infty[ \to \bbr^d$ such that
\[
X_t^x(\omega)=f(x,t)
\]
for every $\omega\in\Omega$. Let us first state two well-known results which we use as a starting point. The first one is taken from \cite{jacodshir} Proposition I.4.28:

\begin{proposition} \label{prop:detsemimg}
Let $f$ be a real-valued function on $[0,\infty[$. The (simple) process $X_t(\omega)=f(t)$ is a semimartingale iff $f$ is c\`adl\`ag and of finite variation on compact intervals.
\end{proposition}

\begin{proposition}
A deterministic process $X^x$ is a L\'evy process iff it can be written as $x+a\cdot t$ with $a\in\bbr^d$.
A one-dimensional deterministic process $X^x$ is a subordinator iff it can be written as $x+a\cdot t$ with $a\geq 0$.
A L\'evy process $X^x$ is deterministic iff it can be written as $x+a\cdot t$ with $a\in\bbr^d$.
\end{proposition}

Our standard reference for results on L\'evy processes and subordinators is \cite{sato}. For a deterministic Markov process the time homogeneity \eqref{timehomone} reads as follows: if there exists $s,t\geq 0$ and $x,y\in\bbr^d$ such that $X_s^x=X_t^y$ we obtain
\begin{align} \label{timehom}
  X_{s+h}^x = X_{t+h}^y
\end{align}
for $h\geq 0$. In the sequel we will first deal with the behavior of a single path $t\mapsto X_t^x$. Directly from \eqref{timehom} we obtain the following.

\begin{proposition} Let $(X^x)_{x\in\bbr^d}$ be a deterministic Markov process and let $x\in\bbr^d$. If there exists $t_0<t_1$ such that $X_{t_0}^x=X_{t_1}^x$ then $X_{t_0+h}^x=X_{t_1+h}^x$ for every $h\geq 0$.
\end{proposition}

\begin{remark} This means that if a path returns to a point, it has visited before, it becomes periodic. Obviously there exists a smallest pair $t_0,t_1$ meeting the requirements of the above proposition. In this case one could speak of a pre-periodic phase up to time $t_0$ and afterwards of periods of length $t_1-t_0$. This proposition as well as the next one remain true for a general state space.
\end{remark}

\begin{proposition} \label{prop:locallyconstant}
If a path of the deterministic Markov process $(X^x)_{x\in\bbr^d}$ becomes locally constant, it remains constant forever, i.e.
if $t\mapsto X_t^x$ is constant on $[t_0,t_1]$ (for some $t_0<t_1$) then $X_t^x=X_{t_0}^x$ for every $t>t_0$.
\end{proposition}

\begin{proof}
If the process is locally constant, there exists an $h>0$ such that $P_t(x,\{x\})=1$ for $t\leq h$. For every $t \geq 0$ there exist $n\in\bbn$ and $0\leq \varepsilon \leq h$ such that $t = n\cdot h + \varepsilon$. We obtain by the Chapman-Kolmogorov equation (see e.g. \cite{ethierkurtz} Formula (4.1.10))
\begin{align*}
  P_t(x,\{x\})=\int_{\bbr^d} ... \int_{\bbr^d} P_\varepsilon(y_n,{x}) P_h(y_{n-1},dy_n)\ ... \  P_h(x,dy_1)=1.
\end{align*}
\end{proof}

The following result is an immediate consequence of the definition of quasi-left continuity:

\begin{proposition}
A deterministic Markov process is Hunt iff its paths are continuous.
\end{proposition}

Now we take the order structure into account and therefore restrict ourselves to Markov processes on $\bbr$.

\begin{theorem} \label{thm:types}
Let $X$ be a one-dimensional deterministic Hunt process. For every path $t\mapsto X_t^x$ there exists a $t_0\in [0,\infty]$ such that $t\mapsto X_t^x$ is strictly monotonically (increasing or decreasing) on $[0,t_0[$ and constant on $[t_0,\infty[$.
\end{theorem}

\begin{remark}
Since $t_0\in [0,\infty]$ the `pure types' of paths which are only constant OR strictly monotone are included.
\end{remark}

\begin{proof}
Let $t\mapsto X_t^x$ be a path which is not of the type described above. In this case there exist $s\leq t\leq u$ such that $X_s^x=X_u^x=:m$ and (w.l.o.g.) $X_t^x > X_s^x$. Let $M$ be the maximum value of the continuous path restricted to the compact set $[s,u]$. Furthermore let $t_{max}$ be the maximum of the set
\[
\{t\in [s,u] : X_t=M \}.
\]
This maximum is attained since the paths are left continuous. Furthermore there exists an $\varepsilon > 0$ such that $B_\varepsilon(t_{max}) \subseteq ]s,u[$. There exists a $t_0\in]t_{max}-\varepsilon/2,t_{max}[$ where a value $m_0\in ]m,M[$ is attained by $t\mapsto X_t^x$. Otherwise the path would be locally constant which leads to a contradiction by Proposition \ref{prop:locallyconstant}. By the intermediate value theorem every point in $[m_0,M]$ is attained in $[t_0,t_{max}]$. In particular an arbitrary value $m_1$ which is attained (again by the intermediate value theorem) in $t_2\in]t_{max},t_{max}+\varepsilon/2[$ is as well attained in $t_1\in ]t_0,t_{max}[\,\subseteq \, ]t_{max}-\varepsilon/2,t_{max}[$. By the definition of $M$ we obtain
\[
P_{t_1,t_{max}}(m_1,\{M\})=1 \neq 0 =P_{t_2,t_2+(t_{max}-t_1)}(m_1,\{M\})
\]
which is a contradiction to time homogeneity.
\end{proof}

In the non-deterministic world there exist Hunt processes which are not semimartingales: Let $W$ be a standard Brownian motion. The process $X:=\abs{W}^{1/2}$ is Hunt without being a semimartingale (see \cite{yor}). However, in the deterministic setting we have the subsequent result which follows directly from Theorem \ref{thm:types}.

\begin{corollary} \label{cor:huntsemimg}
Every one-dimensional deterministic Hunt process is a semimartingale.
\end{corollary}

This corollary does not hold true for general deterministic Markov processes as the following example illustrates:

\begin{example} Let
\begin{align*}
X_t^0(\omega):= \begin{cases} t &\text{if } t\in \bigcup_{n=0}^\infty \left[1-\frac{1}{2^{2n}},1-\frac{1}{2^{2n+1}} \right[\\
-t & \text{if } t\in \bigcup_{n=0}^\infty \left[1-\frac{1}{2^{2n+1}},1-\frac{1}{2^{2n+2}} \right[ \\
0  & \text{if } t\in [1,\infty[
\end{cases}
\end{align*}
This paths looks as follows:

\centerline{
\setlength{\unitlength}{1cm}
\begin{picture}(1,2)
\put(0,0){\vector(0,1){2}}
\put(0,1){\vector(1,0){2.1}}
\linethickness{0.5mm}
\put(0,1){\line(1,1){0.5}}
\put(0.5,0.5){\line(1,-1){0.25}}
\put(0.75,1.75){\line(1,1){0.125}}
\put(0.875,0.125){\line(1,-1){0.0625}}
\put(0.9375,1.9375){\line(1,1){0.03125}}
\put(0.96875,0.03125){\line(1,-1){0.015625}}
\put(1,1){\line(1,0){1}}
\put(1,0.8){1}
\put(2,0.8){2}
\put(-0.2,1.8){1}
\end{picture}
}
Since the path starting in zero crosses the interval $]-1/2,1/2[$ infinitely often on $[1/2,1]$ it is not of finite variation on compacts and the left-hand side limit in 1 does not exist. By Proposition \ref{prop:detsemimg} the process is not a semimartingale. For the other starting points we define the process as follows: if there exists a $t_x\geq 0$ such that $X_{t_x}^0 = x$,
we set $X^x_t:=X_{t_x+t}^0$. If there is no such $t_x$ we set $X^x_t:=x$ for every $t\geq 0$. This gives us a Markov process. Using a slightly different construction one can show that even a deterministic \emph{c\`adl\`ag} Markov process need not be a semimartingale (cf. Example \ref{ex:cadlagmp}).
\end{example}

The statement of Corollary \ref{cor:huntsemimg} does not hold for dimension $d\geq 2$ as it is shown in example \ref{ex:multidim} below. From now on we restrict ourselves to one-dimensional Hunt processes. Next we analyze the interdependence between the paths of such a process. Let us start with some elementary facts: if two paths hit each other, i.e. $X_{t_0}^x=X_{t_0}^y$ then they `stick together':
\[
X_t^x=X_t^y \text{ for every } t\geq t_0.
\]
This is due to \eqref{timehom}.

By Theorem \ref{thm:types} two paths can only hit each other at time $t_0$ if one is strictly increasing up to $t_i\leq t_0$ and constant on $[t_0,\infty[$ and the other one is strictly decreasing up to $t_d\leq t_0$ and afterwards constant (and $t_i = t_0$ or $t_d = t_0$).

\centerline{
\setlength{\unitlength}{1cm}
\begin{picture}(1,2)
\put(0,0){\vector(0,1){2}}
\put(0,1){\vector(1,0){2.1}}
\linethickness{0.5mm}
\put(1,1){\line(1,0){1}}
\put(0,0){\line(1,1){1}}
\put(0,1.5){\line(2,-1){1}}
\put(1,0.6){$t_0$}
\end{picture}
}

Let us start with an $x_0\in\bbr$ such that $X^{x_0}$ is increasing (at least for an initial period of time). There are three possibilities how the path can behave:
case a: it grows up to infinity (if it does so in finite time, we have a killing).
case b: it is everywhere strictly monotone, i.e. $t_0=\infty$ in Theorem \ref{thm:types}, but it is bounded and its $\limsup$ is $y$.
case c: it is strictly monotone up to time $t_0$ and afterwards it is constantly equal to $y$.
In the respective cases we know the behavior of the paths starting in $[x_0,\infty[$, $[x_0,y[$ or $[x_0,y]$ by formula \eqref{timehom}.
Since the process is increasing (for an initial time period) on the intervals $[x_0,\infty[$, $[x_0,y[$ resp. $[x_0,y[$ we call them a $\oplus$-domain. An interval on which the paths are decreasing (for an initial time period) is called $\ominus$-domain and a (possible degenerate) interval on which the paths are constant is called $\odot$-domain. In the case c above $\{y\}$ already belongs to a $\odot$-domain above the first interval (see below).

The $x_0$ was chosen arbitrarily. Therefore we should now examine what happens for $x<x_0$. Two things can happen: either there exists an $h>0$ such that
\[
X_{t+h}^x = X_t^{x_0}
\]
(in this case we have $X_{t+h+s}^x = X_{t+s}^{x_0}$ by formula \eqref{timehom} for every $s\geq 0$) or not. Now we set:
\begin{align} \label{inf}
  z:=\inf \{x\leq x_0 : \text{ there exists an } h>0 \text{ such that } X_{t+h}^x = X_t^{x_0} \}
\end{align}
$z$ could be $-\infty$. Otherwise we have again three possibilities:
\[
\left.
\begin{array}{l}
\text{case 1: for $x=z$ there still exists such an } h>0 \text{ and we are still in the } \oplus \text{-domain.}\\
\text{case 2: the point belongs to a } \odot \text{-domain.}\\
\text{case 3: the point belongs to a } \ominus\text{-domain.}
\end{array}
\right\} (\star)
\]
In any case the $\oplus$-domain ends here and below we have either a $\odot$- or a $\ominus$-domain or another $\oplus$-domain.

\begin{center}
\includegraphics[width=30mm, angle=-90]{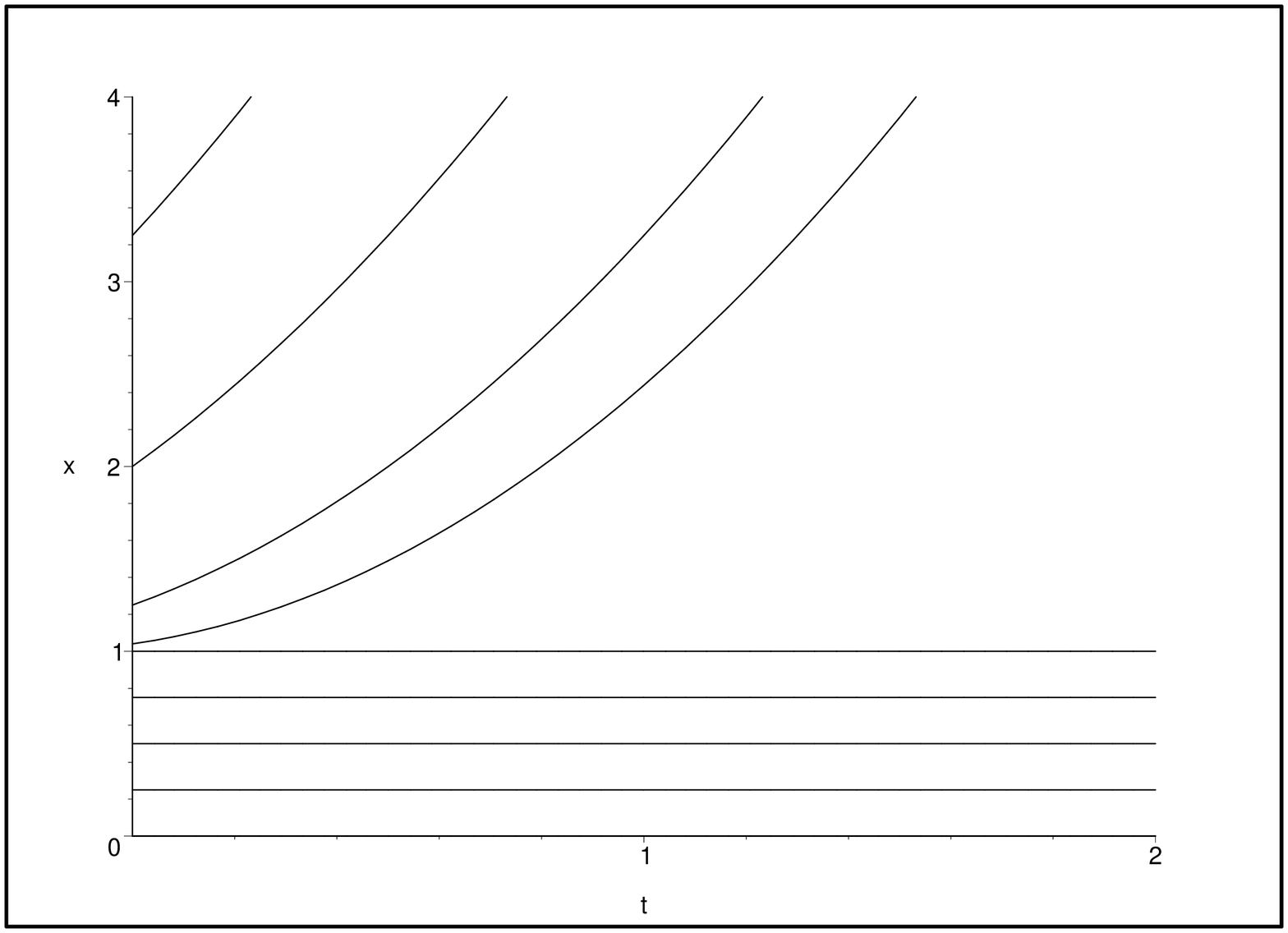}
\includegraphics[width=30mm, angle=-90]{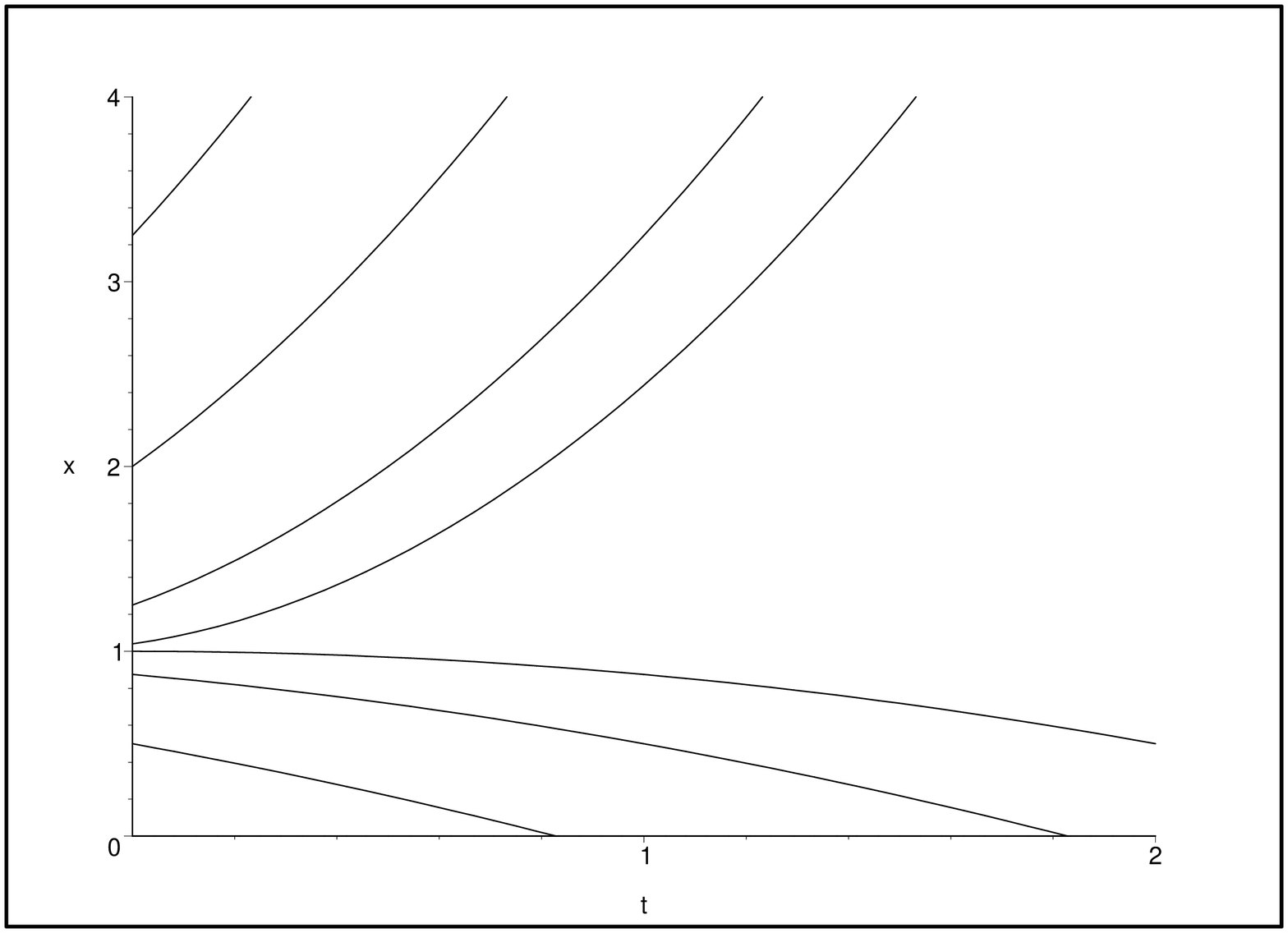}
\includegraphics[width=30mm, angle=-90]{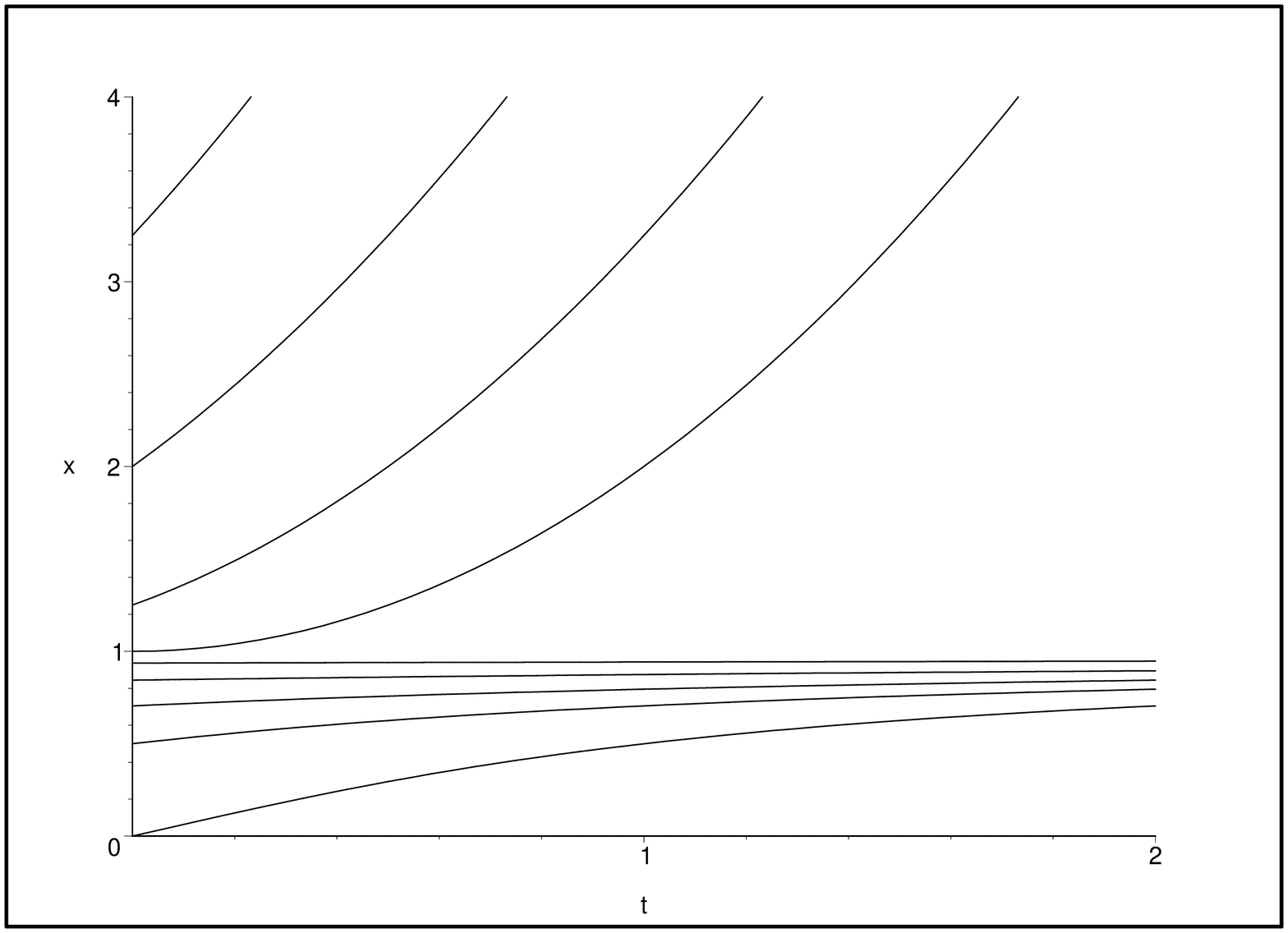}
\end{center}

For the $\oplus$-domain we have just analyzed there is a generating path as in Example \ref{ex:startingpoint}: we set $\Phi(t):= X_t^{x_0}$ for $t\in[0,t_0[$ with $t_0$ as in Theorem \ref{thm:types}. Furthermore let $h_{max}$ be the supremum of $h>0$ appearing in \eqref{inf}. Then we set for $t\in ]-h_{max},0[$ (in case 1 above the left endpoint is included) $\Phi(t):=x$ where $x$ satisfies $X_{-t}^x=x_0$. The point $x$ is unique since the paths are strictly monotonically increasing. For the interval $[-h_{max},t_0[$ resp. $]-h_{max},t_0[$ we write $I$.

Now we analyze the structure of the process step-by-step. To this end we have to consider 6 cases: starting from a $\oplus$, $\ominus$ or $\odot$-domain we can proceed upwards and downwards. W.l.o.g. we will go upwards. For the other three case one has just to interchange the r\^oles of $\oplus$ and $\ominus$ and of $\infty$ and $-\infty$.

We start with the $\oplus$-domain (cases a,b,c from above): In the first case we are done. There is no interval above the one we considered. In case b the behavior in $y$ is not known. Either there starts another $\oplus$-domain or a $\odot$-domain. In case c we already know that a $\odot$-domain starts which might consist of a single point.

Analyzing a $\odot$-domain the upper endpoint $y$ can be $\infty$; in this case we are done. The upper endpoint might belong to the $\odot$-domain. In this case we can continue with either a $\oplus$ or a $\ominus$-domain. In the second case the paths in the $\ominus$-domain can reach $y$ or not. If the upper endpoint $y<\infty$ belongs to the $\odot$-domain, we have to continue with a $\oplus$ domain.

In the case of the $\ominus$-domain we have to proceed as in the above consideration leading to ($\star$). We set
\begin{align}
  y:=\sup \{x\geq x_0 : \text{ there exists an } h>0 \text{ such that } X_{t+h}^x = X_t^{x_0} \}.
\end{align}
If $y=\infty$ we are done. Otherwise this upper endpoint might still belong to the $\ominus$-domain or not. Above we can have a $\oplus$ or a $\odot$-domain or again a $\ominus$-domain.

These considerations lead to the following classification theorem:

\begin{theorem} \label{thm:structure}
A family of functions $t\mapsto X^x_t$, each mapping $[0,\infty[$ into $\bbr$, is a deterministic Hunt process if and only if there exists a decomposition of $\bbr$ into disjoint ordered intervals $(J_j)_{j\in Z}$ where $Z\subset \{-n,...,0,...,m\}$ with $n,m\in\bbn\cup \{\infty \}$ such that
on every Interval $J_j$ the functions $t\mapsto X^x_t$ (for $x\in J_j$) are either all constant or there exists a continuous function $\Phi_j:I_j\to J_j$ called the generating path for $J_j$ which is surjective and either strictly monotonically increasing or strictly monotonically decreasing and such that
  \[
    X_t^x=\Phi_j(t+\Phi_j^{-1}(x)) \text{ for } x\in J_j \text{ and } t\in [0,\infty[ \, \cap \, (I_j-\Phi_j^{-1}(x))
  \]
and the $I_j$ are intervals containing zero.
\end{theorem}

\begin{definition}
With every interval $J_j$ we associate the \emph{type} $\oplus$, $\ominus$ resp. $\odot$, if $\Phi$ is increasing, decreasing resp. the process is constant on the interval. This allows to describe the process from an abstract point-of-view as a sequence like $...|\odot|\oplus|\odot|\ominus|...$. This sequence is called the \emph{structure} of the process.
To emphasize that the upper endpoint of the lower interval belongs to the lower (resp. higher) interval we write $\odot]\oplus$ (resp. $\odot[\oplus$).
\end{definition}

\begin{remarks} \label{rem:important}
a) Obviously the behavior of the paths of the Hunt process is totally described by the decomposition $(J_j)_{j\in Z}$ and the sequence of generating paths. One has to observe that a path starting in a $\oplus$-domain can reach the lowest point of a $\odot$-domain becoming constant and the same is true for a path starting in a $\ominus$-domain reaching the highest point of a $\odot$-domain. If we want to emphasize this we write $\oplus\hspace{0.5mm}[\hspace{-2mm}\rightarrow \odot$ resp. $\odot\hspace{1.5mm}]\hspace{-4mm}\leftarrow \ominus$.

b) If $J_j$ is a $\oplus$-domain then $J_{j+1}$ can not be a $\ominus$-domain. Between the two there has to be a $\odot$ (which can be of course degenerate, i.e. consisting of only one point).

c) Using the second convention of the above definition we obtain that $\oplus]\oplus$ and $\oplus]\odot$ are not allowed. The paths of the right endpoint of the lower domain has to be strictly monotonically increasing. Writing $\Phi_1:I_1\to J_1$ for the generating path of the lower interval, $\oplus[\oplus$ and $\oplus[\odot$ are allowed but make only sense if the right endpoint of the interval $I_1$ is $\infty$. An analogous statement holds for $\ominus|\ominus$ and $\odot|\ominus$.

d) In every $\oplus$- and $\ominus$-domain there exists one unique point $x_j$ such that $\Phi^{-1}(x_j)=0$.

e) Consider $x\to \infty$ (resp. $x\to -\infty$). Either $m=\infty$ (resp. $n=-\infty$) or there exists a highest (lowest) interval. If furthermore $J_m$ ($J_n$) is of the type $\oplus$ (resp. $\ominus$), the right endpoint of $\Phi_m$ ($\Phi_n$) has to be $\infty$. If this was not the case we would introduce a killing.

f) To get a unique representation: Plug together $\odot$ intervals if they follow each other, i.e. do not allow $\odot|\odot$. Chose always the middle point of an interval as the one $x_j$ with $\Phi^{-1}(x_j)=0$ for ever $\oplus$- and $\ominus$-domain; except for the lowest/highest interval: if there is a lowest interval $J_n=]-\infty,b]$ or $]-\infty,b[$, set $x_j:=b-1$. If there is a highest interval $J_m=[a,\infty[$ or $]a,\infty[$, set $x_j:=a+1$. And finally claim $0 \in J_0$.

g) Only the following types of intervals $I_j$ appear ($a,b\in\bbr$): $]-\infty,\infty[$, $]-\infty,b[$, $]a,\infty[$, $]a,b[$, $[a,\infty[$ and $[a,b[$.
\end{remarks}

Occasionally we will write
\[
  \Phi_\oplus: I_\oplus \to J_\oplus \text{ and } \Phi_\ominus: I_\ominus \to J_\ominus
\]
if we want to emphasize the type of the generating path rather then the relative position of $J_j$.

\section{Characterization of Some Subclasses}

In this section we characterize when a deterministic Hunt process is Feller, $C_b-$Feller and rich. Furthermore we calculate the symbol of a deterministic Feller process. Again we restrict ourselves to one-dimensional processes.

Let us start with the Feller property:

\begin{lemma} \label{lem:cont}
Let $X$ be a deterministic Hunt process. The function $x\mapsto
T_tu(x)$ is continuous for every $t\geq 0$ and $u\in C(\bbr)$ if and
only if the process is of pure type or if the structure consists
only of the following building blocks:
\begin{align} \label{fellerstructures}
\oplus|\odot, \odot|\ominus
\end{align}
or $\odot|\oplus$, $\ominus|\odot$ if in these two cases the left
endpoint of $I_\oplus$ resp. $I_\ominus$ is $-\infty$.
\end{lemma}

\begin{proof}
Let $X$ be of one of the structures prescribed in the lemma. We
consider w.l.o.g. the case $\oplus|\odot$. Let $t\geq 0$ be fixed.
In the interior of each domain we have continuity, since
\[
T_tu(x)=u(X_t^x)=u\Big(\Phi_\oplus(t+\Phi_\oplus^{-1}(x))\Big)
\]
which is a composition of continuous functions and the case $\odot$ is trivial. Now we have to deal with the endpoints of the intervals. In any case the $\odot$-domain is a closed interval (cf. Remark \ref{rem:important} c)) . Without loss of generality, let $(x_n)_{n\in\bbn} \subseteq J_\oplus$ be a sequence in the $\oplus$-domain tending to the lower endpoint of $J_\odot$. We have to consider two cases:
\begin{center}
\includegraphics[width=30mm, angle=-90]{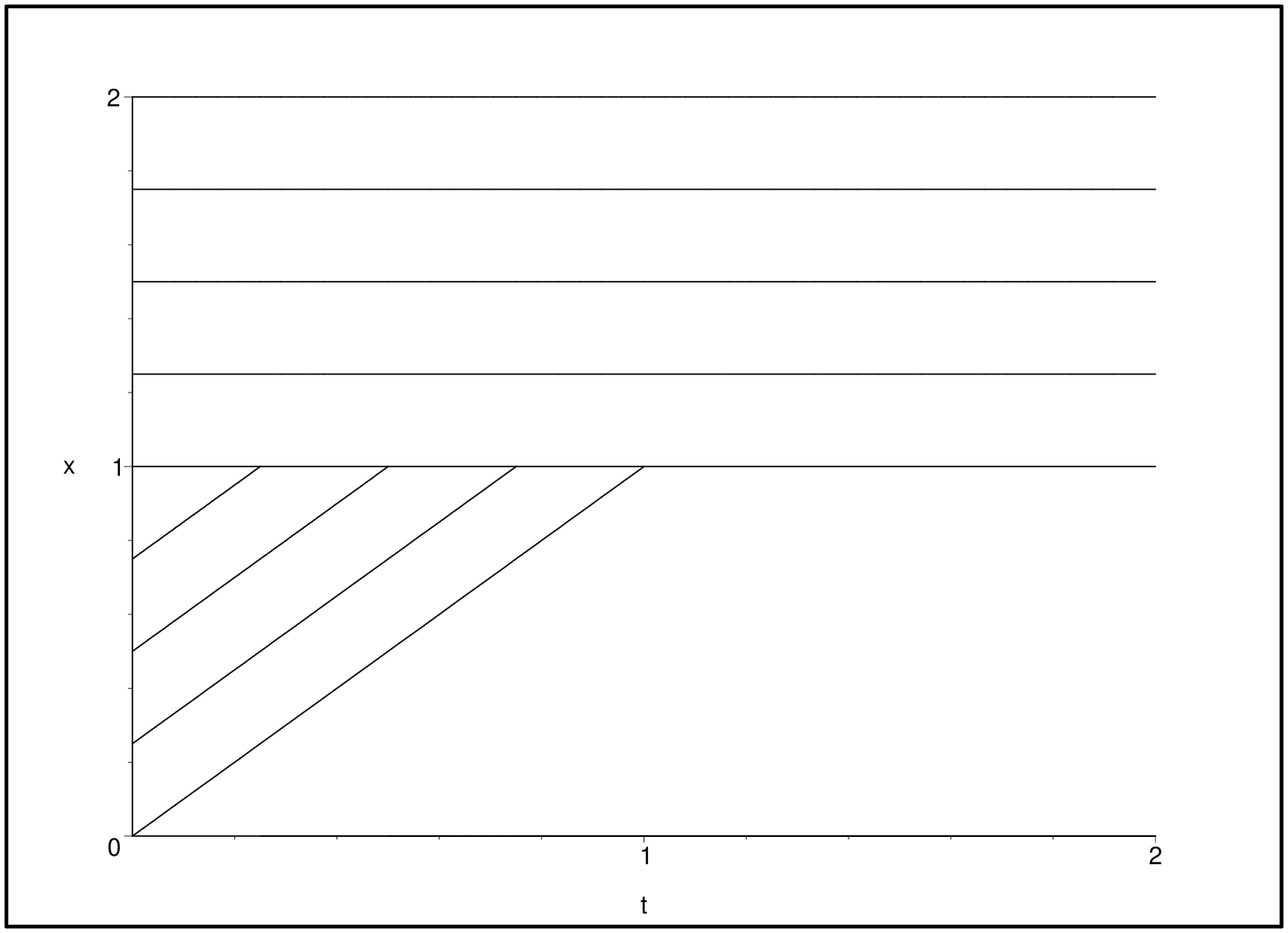}
\includegraphics[width=30mm, angle=-90]{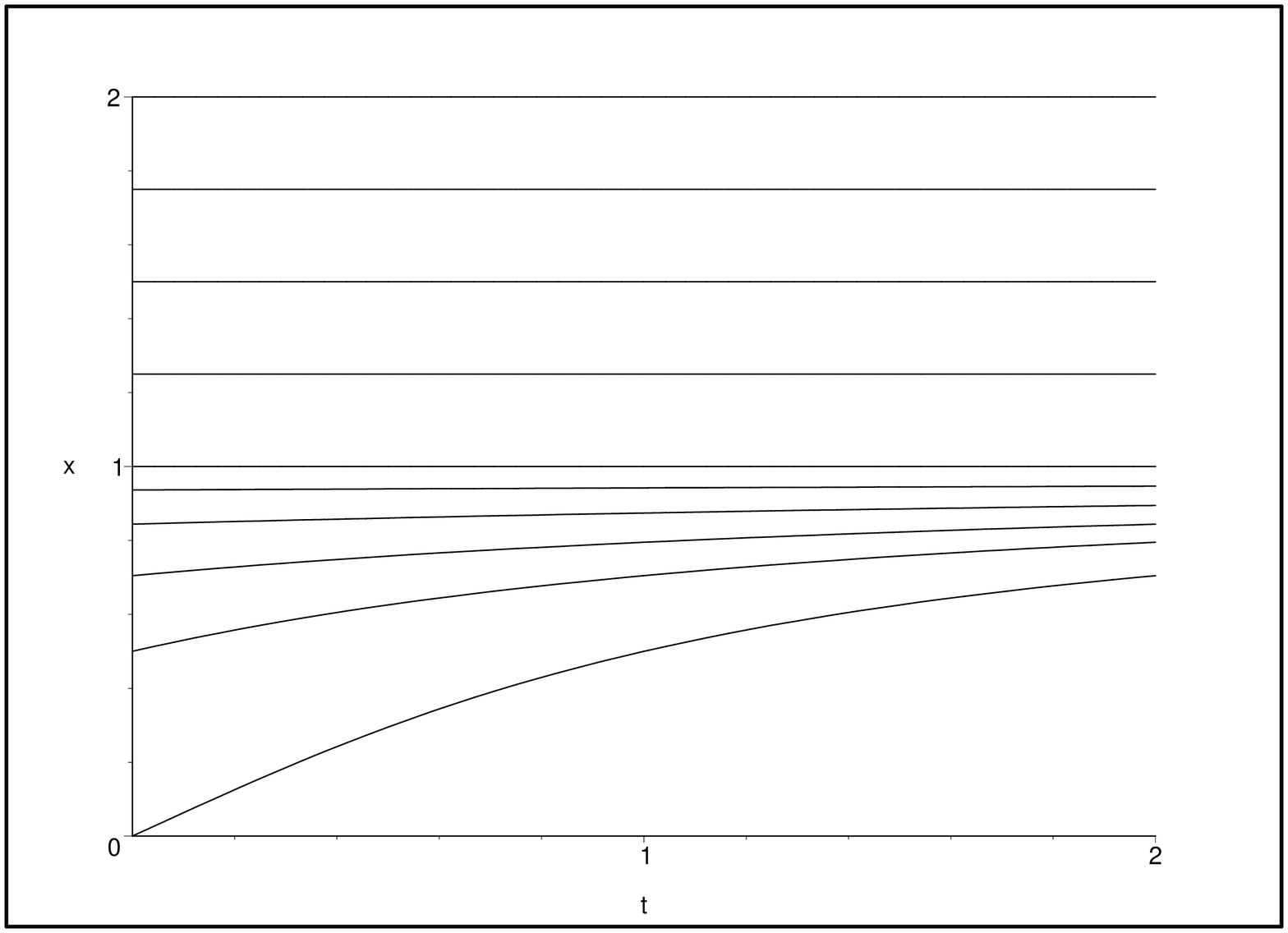}
\end{center}
In the case $\oplus\hspace{0.5mm}[\hspace{-2mm}\rightarrow \odot$ there exists a point $w<x$ such that for $x_n\in [w,x]\subseteq J_\oplus$ we have $X_t^{x_n}=x$. This implies continuity. If the right endpoint of $I_\oplus$ is $\infty$, i.e. the paths starting in $J_\oplus$ do not reach $J_\odot$, we have
\[
T_tu(x_n)=u\Big(\Phi_\oplus(\underbrace{t+\Phi_\oplus^{-1}(x_n)}_{\to \infty})\Big) \to u(x).
\]
It remains to show that indeed only the `building blocks' above are
allowed in order to obtain continuity. Each of the other blocks
$\oplus|\oplus$, $\ominus|\oplus$, $\ominus|\ominus$ as well as
$\odot|\oplus$, $\ominus|\odot$ (not fulfilling the restriction of
the lemma) leads to a contradiction, if it appears in the structure
of the process. Since the reasoning is always quite similar we only
consider the case $\oplus|\oplus$ (i.e. $\oplus[\oplus$ by Remark
\ref{rem:important} c)):
let $t>0$ and let $J_1,J_2$ be two neighboring $\oplus$-domains. Let $x$ be the lower endpoint of $J_2$. Let $u\in C_\infty(\bbr)$ such that it is the identity in a neighborhood $U$ of $x$. Let $(x_n)_{n\in\bbn}\subseteq U\cap J_1$ such that $x_n\to x$. As in the considerations above we have $T_tu(x_n) \to u(x)=x$, but
\[
T_t(x)=\Phi_2\Big(\underbrace{t}_{>0}+\Phi_2^{-1}(x)\Big)>x
\]
since $\Phi_2$ is strictly monotonically increasing.
\end{proof}

\begin{lemma}
Let $X$ be a deterministic Hunt process which is of one of the
structures described in Lemma \ref{lem:cont}. In this case,
$x\mapsto T_tu(x)$ is vanishing at infinity for every $t\geq 0$ if
and only if (i) $n=\infty$ or $I_{-n}$ is not a $\oplus$-domain or
the left endpoint of $I_{-n}$ is $\infty$ and (ii) $m=\infty$ or
$I_{m}$ is not a $\ominus$-domain or the left endpoint of $I_{m}$ is
$\infty$.
\end{lemma}

\begin{proof}
The case of the pure $\odot$-type is trivial. The right endpoint of $I_\ominus$ and $I_\oplus$ is always $\infty$, since we do not consider processes with killing (cf. Remark \ref{rem:important} e)).  W.l.o.g. we consider $x\to -\infty$ in the $\oplus$-domain:
\[
T_tu(x)=u\Big(\Phi_\oplus(\underbrace{t+\Phi_\oplus^{-1}(x)}_{\to t+a})\Big)
\]
If $a=-\infty$, this expression tends to zero since $u\in C_\infty(\bbr)$. If not, there exists a $u\in C_\infty(\bbr)$ such that $u(\Phi_\oplus(t+a))\neq 0$.
\end{proof}

The following lemma is a reformulation of \cite{revuzyor} Proposition III.2.4.
\begin{lemma}
Let $X$ be a Hunt process which satisfies $(F1)$. In this case $(F2)$ is equivalent to
\begin{center}
$(F2)^*$ \hspace{5mm} $T_tu(x)\xrightarrow[]{t \to 0} u(x)$ for every $u\in C_\infty(\bbr^d)$ and every $x\in\bbr$.
\end{center}
\end{lemma}

\begin{lemma}
Let $X$ be a deterministic Hunt process satisfying $(F1)$. In this case $(F2)^*$ holds.
\end{lemma}

\begin{proof}
Let $x\in\bbr$. If $x\in J_\odot$ the statement is trivial, if not, we have to consider
\[
T_tu(x)=u\Big(\Phi_j(t+\Phi_j^{-1}(x))\Big)\xrightarrow[t\downarrow 0]{} u(x), \hspace{10mm} j\in\{\oplus,\ominus\}.
\]
\end{proof}

Putting the results of Lemmas \ref{lem:cont} - 3.4 together we obtain the follwing result:

\begin{theorem} \label{thm:feller}
Let $X$ be a deterministic Hunt process. $X$ is Feller if and only
if it is of one of the structures described Lemma \ref{lem:cont} and
writing
\[
\Phi_{-n}:I_{-n} \to J_{-n} \text{ resp. } \Phi_m:I_m\to J_m
\]
the left endpoints of $I_{-n}$ and $I_m$ are $-\infty$.
\end{theorem}

The restriction on $I_{-n}$ and $I_m$ are of course only necessary
if there is a lowest $\oplus$- resp. highest $\ominus$-domain. For
$C_b$-Feller processes we obtain the following:

\begin{theorem}
Let $X$ be a deterministic Hunt process. $X$ is $C_b$-Feller if and
only if it is of one of the structures described in Lemma
\ref{lem:cont}.
\end{theorem}

\begin{proof}
By Lemma \ref{lem:cont} we obtain that for every $t\geq 0$ and $u\in
C_b(\bbr)$ the function $x\mapsto T_tu(x)$ is continuous. Since
$T_tu(x)= u(\Phi_j(t+\Phi_j^{-1}(x)))$ on the $\oplus$- and
$\ominus$-domains and $T_tu(x)=u(x)$ on $J_\odot$, the function is
bounded. It remains to show locally uniform convergence at zero: We
only consider the most difficult structure $\oplus|\odot|\ominus$.
Let $\varepsilon>0$. Let $[a,b]:=J_\odot$. On the interval
$[a-\varepsilon,b+\varepsilon]$ we have $|X_t^x-x|\leq \varepsilon$
since $\Phi_\oplus$ (resp. $\Phi_\ominus$) is strictly monotonically
increasing (resp. decreasing). Furthermore $u$ is uniformly
continuous on this interval. For $x\notin
[a-\varepsilon,b+\varepsilon]$ we argue as follows: w.l.o.g. we
consider $]-\infty,a-\varepsilon]\subseteq J_\oplus$. Choose $t_0$
such that for $0\leq t \leq t_0$ and $x\in ]-\infty,a-\varepsilon]$
we have
\[
X_t^x\in \left]-\infty,a-\frac{\varepsilon}{2}\right].
\]
Let $w<a-\varepsilon$. Then $\Phi_\oplus^{-1}([w,a-\varepsilon])+[0,t_0]$ is a compact set on which the function $u\circ \Phi_\oplus$ is uniformly continuous. Since $w$ was chosen arbitrarily, the result follows.
\end{proof}

We close this section by dealing with the generator of the process which is defined as follows:

\begin{definition} \label{def:generator}
The \emph{generator} $A$ of the semigroup $(T_t)_{t\geq 0}$ is the linear mapping $A:D(A) \to B_b(\bbr^d)$:
\[
A u := \lim_{t \downarrow 0} \frac{T_t u - u}{t} \hspace{1cm} (u \in D (A))
\]
where
\[
D(A):=\left\{u\in B_b(\bbr^d) : \lim_{t \downarrow 0} \frac{T_t u - u}{t} \text{ exists in } \norm{\cdot}_\infty\right\}
\]
is the \emph{domain} of the operator.
\end{definition}

A Hunt process is called \emph{rich}, if the test functions $C_c^\infty(\bbr)$ are contained in $D(A)$ and $A(C_c^\infty(\bbr))\subseteq C_\infty(\bbr)$. A classical result which is due to P. Courr\`ege (cf. \cite{courrege}) tells us that if $X$ is a rich Feller process, the generator is a pseudo-differential operator which can be written as
\[
Au(x)=\int_{\bbr^d} e^{ix\xi} p(x,\xi) \widehat{u}(\xi) \, d\xi \hspace{10mm} (u\in C_c^\infty(\bbr))
\]
where $\widehat{u}(\xi)=\frac{1}{(2 \pi)}\int_{\bbr} e^{-ix \xi} u(x) \, dx$ denotes the Fourier transform and
\[
p(x,\xi)= -i \ell(x)  \xi + \frac{1}{2} Q(x) \xi^2
    - \int_{y\neq 0} \Big(e^{i y \xi } -1 - i y \xi \cdot \chi(y)\Big) \, N(x,dy)
\]
is for every fixed $x\in\bbr$ a continuous negative definite function in the sense of Schoenberg (cf. \cite{bergforst} Chapter II). $p:\bbr\times\bbr\to\bbc$ is called the \emph{symbol} of the process. In \cite{mydiss} we have shown that the symbol can be calculated by the formula
\begin{align} \label{myformula}
  p(x,\xi)=- \lim_{t\downarrow 0}\bbe^x \frac{e^{i(X^\sigma_t-x)\xi}-1}{t}
\end{align}
where $\sigma$ is the first-exit-time of an arbitrary compact neighborhood of $x$. This formula allows to generalize the notion of the symbol from Feller to It\^o processes.

\begin{proposition} \label{prop:rich}
Let $X$ be a Feller process of type $\oplus$ or $\ominus$. $X$ is rich if and only if the generating path $\Phi$ is continuously differentiable.
\end{proposition}

\begin{proof}
Let $u\in C_c^\infty(\bbr)$.
If $t\mapsto X_t^x$ is differentiable from the right in zero we have
\[
Au(x)= \lim_{t\downarrow 0}\frac{u(X_t^x)-u(X_0^x)}{t}=\left.\frac{\partial}{\partial t}\right|_{t=0}^+ (u\circ X_t^x),
\]
if not, the limit does not exist and hence the test functions are not contained in $D(A)$. By the representation using the generating path $\Phi$ we obtain
\begin{align*}
  Au(x)&=\left. u'(\Phi(t+\Phi^{-1}(x)))\cdot \Phi'(t+\Phi^{-1}(x))\cdot 1 \right|_{t=0}\\
       &= u'(x) \cdot \Phi'(\Phi^{-1}(x))
\end{align*}
where $\cdot'$ denotes the right hand side derivative which coincides with the derivative if it is continuous. We already know that $\Phi$, $\Phi^{-1}$ and $u'$ are continuous. Therefore, in order to obtain continuity of $x\mapsto Au(x)$ it is necessary and sufficient that $\Phi'$ is continuous. It is a well-known fact that continuity of the right-hand side derivative implies that the function is continuously differentiable. Since the support of $u'$ is compact, $x\mapsto Au(x)$ is automatically vanishing at infinity.
\end{proof}

\begin{remark}
If we are not in a `pure type' case as in the proposition we need to claim that every $\Phi_j$ is differentiable and furthermore there has to be a smooth transition from every $J_j$ to the neighboring intervals (cf. Example \ref{ex:middlederivativezero}).
\end{remark}

\begin{theorem}
Let $X$ be a deterministic rich Feller process. In a $\odot$-domain the symbol $p(x,\xi)$ is zero. In a $\oplus$- or $\ominus$-domain the symbol is $i\xi \Phi_j'(\Phi_j^{-1}(x))$.
\end{theorem}

\begin{proof}
Follows directly from \eqref{myformula}.
\end{proof}

\section{The Cantor Process}

Now we use our previous results in order to prove the existence of a Hunt semimartingale which is not an It\^o process. The only known example of such a process is the absolute value of a Brownian motion. Compare in this context \cite{cinlarjacod81} Example (3.58). Our process has the advantage of being defined on the whole real axis.

Let $C$ be the Cantor set and $h:\bbr\to [0,1]$ be the Cantor function (cf. \cite{cantorfunction} and \cite{elstrodt} Section 8.4) and define $g:[0,1]\to [0,1]$ by $g(y):= (1/2)(h(y)+y)$.

By the well known results on $h$ we obtain the following properties of $g$:
\begin{itemize}
  \item $g$ is strictly monotonically increasing.
  \item $g(0)=0$ and $g(1)=1$
  \item It is continuous and bijective.
  \item It is differentiable in $[0,1]\backslash C$ and the derivative in these points is $1/2$.
\end{itemize}

\begin{definition}
For $x\in\bbr$ let $\Phi(x):=g(x-[x])+[x]$ with $g$ defined above and $x\mapsto[x]$ denoting the floor funcion. The stochastic process $X=(X_t)_{t\geq 0}$ defined by
\[
  X_t^x:=\Phi(t+\Phi^{-1}(x))
\]
is called \emph{Cantor process}.
\end{definition}

Since $\Phi:\bbr\to\bbr$ is strictly monotonically increasing, continuous and bijective, $X$ is a Hunt process (Theorem \ref{thm:structure}) and by Corollary \ref{cor:huntsemimg} it is a semimartingale.  Being defined by a single increasing $\Phi$ its structure is $\oplus$, in particular it is of `pure type'. $X$ is a Feller process by Theorem \ref{thm:feller}, because $\Phi$ is defined on the real line.

\begin{proposition}
The Cantor process $X$ is not an It\^o process.
\end{proposition}

\begin{proof}
Let $(B,C,\nu)$ denote the semimartingale characteristics of $X$. $C$ and $\nu$ are zero since the paths of the process are continuous and of finite variation on compacts. If $X$ was an It\^o process, there would exist a measurable $\ell:\bbr\to\bbr$ such that
\[
X_t^0= B_t=\int_0^t \ell(X_s) \, ds.
\]
Restricted to $t\in[0,1]$ this means
\[
 \int_0^t  \ell(X_s) \, ds = g(t) = \frac{1}{2} (h(t) + t) = \frac{1}{2}h(t) + \int_0^t \frac{1}{2} \, ds
\]
and therefore $h$ has the Lebesgue density $2 \ell(X_s)-1$. This contradicts the well known fact that the Cantor function does not admit a Lebesgue density.
\end{proof}

\begin{corollary}
The Cantor process is not rich.
\end{corollary}

\begin{proof}
We have already seen that $X$ is a Feller process. If it was rich, we would obtain by Theorem 3.10 of \cite{mydiss} that it is an It\^o process. This is a contradiction to the proposition above.
\end{proof}

\begin{remark}
By Theorem 3.35 of \cite{cinlarjacod81} every Hunt process can be written as a random time change of an It\^o process. In our case this (random) time change can be chosen to be $A:t\mapsto t+\Phi^{-1}(x)$ which leads to $Y^x_{A(u)} = X_u^x$ for $u\geq 0$ and $Y_t^x=x+t$, i.e. the It\^o process from which $X$ is obtained is in this case just a deterministic L\'evy process.
\end{remark}

\section{Further Examples}

\begin{example}\label{ex:knickprozess}
With the notation of the previous sections we define $\Phi(y):=\Phi_0(y):= (1/2)y\cdot 1_{]-\infty,0[}(y)+y\cdot 1_{[0,\infty[}(y)$. Obviously we have $I_0=J_0=\bbr$ and $\Phi^{-1}(z)= 2z\cdot 1_{]-\infty,0[}(z)+z\cdot 1_{[0,\infty[}(z)$. The corresponding Markov process is given by
\[
X_t^x:=\Phi\Big(t+\Phi^{-1}(x)\Big)
\]
The process is Feller by Theorem \ref{thm:feller}, but not rich (cf. Proposition \ref{prop:rich}). Nevertheless it is an It\^o process with first characteristic
\[
B_t^x=X_t^x-x=\int_0^t b(X_t^x) \, ds
\]
where $b(y)= (1/2)\cdot 1_{]-\infty,0[}(y)+ \cdot 1_{[0,\infty[}(y)$. The symbol of this process is
\[
p(x,\xi)=-ib(x)\xi
\]
with the same function $b$.
\end{example}

\begin{example}\label{ex:spacedepdrift}
Consider the transition semigroup $(P_t)_{t\geq 0}$ of the \emph{space dependent drift}
\[
P_t(x,B)=\begin{cases}
  1_{B-t}(x) & \text{if } x>0 \\
  1_{B}(x)   & \text{if } x=0 \\
  1_{B+t}(x) & \text{if } x<0,
  \end{cases}
\]
where $t\geq0$, $x\in\bbr$ and $B\in\cb$. The corresponding process is not Feller, but it is an It\^o process (cf. \cite{mydiss} Example B.6). The semimartingale characteristics are $(B,C,\nu)=(X^x-x, 0,0)$ and the symbol of the process is
\[
p(x,\xi)= -i \cdot \text{sign}(x) \xi.
\]
This symbol is not continuous in $x$, but it is finely continuous (cf. \cite{blumenthalget} Section II.4), since $t\pfe\ell(X_t)$ is right continuous for every $\bbp^x \ (x\in\bbr)$.
\end{example}

\begin{example} \label{ex:partlycantor}
Taking the $\Phi$ of the Cantor process (Section 4) as generating path on $]0,\infty]$ and $x\mapsto-x$ on $]-\infty,0]$ we obtain a process which is neither Feller nor nice nor It\^o, but still a Hunt semimartingale.
\end{example}

\begin{example} \label{ex:middlederivativezero}
Let $\Phi_\oplus:[0,\infty[\to[1,\infty[$ be a strictly increasing function which is continuously differentiable with right hand side derivative 0 at zero. Let $\Phi_\ominus:[0,\infty[\to]-\infty,-1]$ be a strictly decreasing function which is continuously differentiable with right hand side derivative 0 at zero. And let $]-1,1[$ be a $\odot$-domain. The process given by this structure is not Feller, but it is a rich It\^o process:
\begin{center}
\includegraphics[width=30mm, angle=-90]{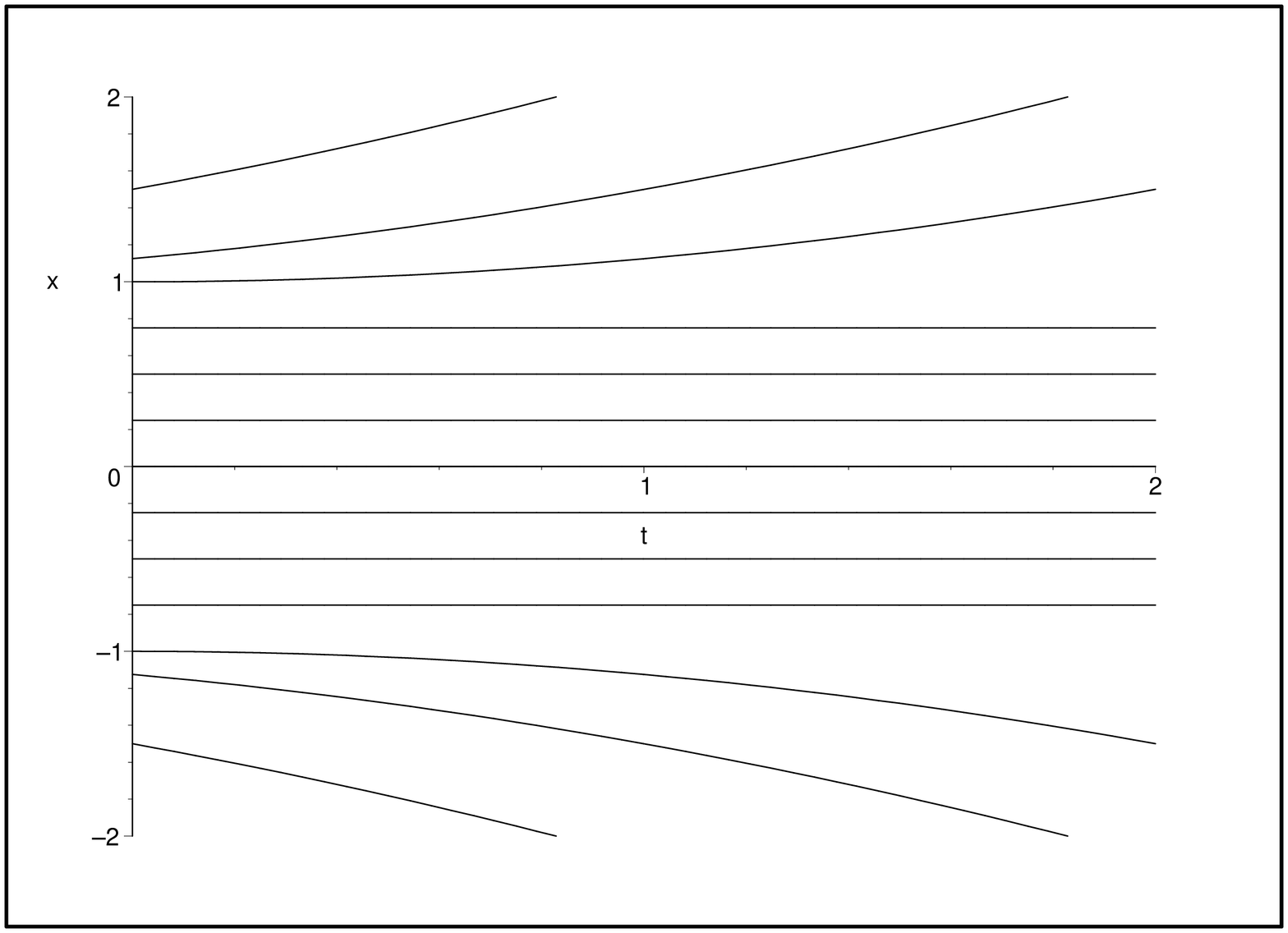}
\end{center}
\end{example}

\begin{example} \label{ex:cadlagmp}
Consider the following path starting in $-1$:
\[
X_t^{(-1)}=\begin{cases} -1+x & \text{if } t\in A\\
                        1-x & \text{if } t\in B\\
                        0   & \text{if } t>1 \end{cases}
\]
where
\begin{align*}
A&=\bigcup_{n=0}^\infty \bigcup_{k=0}^{2^n-1} \left[ \frac{2^n-1}{2^n}+(2k)\frac{1}{4^{n+1}},
                                                     \frac{2^n-1}{2^n}+(2k+1)\frac{1}{4^{n+1}}\right[\\
B&=\bigcup_{n=0}^\infty \bigcup_{k=0}^{2^n-1} \left[ \frac{2^n-1}{2^n}+(2k+1)\frac{1}{4^{n+1}},
                                                     \frac{2^n-1}{2^n}+(2k+2)\frac{1}{4^{n+1}}\right[.
\end{align*}
Since the definition is rather involved, we plot the following diagram for the readers convenience:

\centerline{
\setlength{\unitlength}{2cm}
\begin{picture}(1,2)
\put(0,0){\vector(0,1){2}}
\put(0,1){\vector(1,0){2.1}}
\linethickness{0.5mm}
\put(0,0){\line(1,1){0.25}}
\put(0.25,1.75){\line(1,-1){0.25}}
\put(0.5,0.5){\line(1,1){0.0625}}
\put(0.625,0.625){\line(1,1){0.0625}}
\put(0.5625,1.4375){\line(1,-1){0.0625}}
\put(0.6875,1.3125){\line(1,-1){0.0625}}
\put(0.75,0.75){\line(1,1){0.015625}}
\put(0.78125,0.78125){\line(1,1){0.015625}}
\put(0.8125,0.8125){\line(1,1){0.015625}}
\put(0.84375,0.84375){\line(1,1){0.015625}}
\put(0.765625,1.234375){\line(1,-1){0.015625}}
\put(0.796875,1.203125){\line(1,-1){0.015625}}
\put(0.828125,1.171875){\line(1,-1){0.015625}}
\put(0.859375,1.140625){\line(1,-1){0.015625}}
\put(1,0.8){1}
\put(2,0.8){2}
\put(-0.2,1.8){1}
\put(-0.2,0){1}
\end{picture}
}
For every $n\in\bbn$ the path crosses the interval $[-1/2^n,1/2^n]$ at least $2^{n-1}$ times. Therefore it can not be of finite variation. For the other starting points we define the process as follows: if there exists a $t_x\geq 0$ such that $X_{t_x}^{(-1)} = x$,
we set $X^x_t:=X_{t_x+t}^{(-1)}$. If there is no such $t_x$ we set $X^x_t:=x$ for every $t\geq 0$. By Proposition \ref{prop:detsemimg} the process is not a semimartingale.
\end{example}

\begin{example} \label{ex:multidim}
Next we show that not every deterministic Hunt process is a semimartingale if the dimension of the state-space is bigger than one: let $\Phi:[0,\infty[\to\bbr^2$ be given by
\[
\Phi(t)=\binom{f(t)}{t}
\]
where $f(t)$ is a function of infinite variation on compacts. Furthermore, if there exists a $t_x\geq 0$ such that
\[
\Phi(t_x)=\binom{x_1}{x_2} =: x
\]
we set $X^x_t:=\Phi(t_x+t)$. If there is no such $t_x$ we set $X^x_t:=x$ for every $t\geq 0$. Obviously this process is Hunt. But it is not a semimartingale since for the starting point $(0,0)$ the path is not of finite variation on compacts (cf. Proposition \ref{prop:detsemimg}).
\end{example}

\textbf{Acknowledgements:} I would like to thank my colleague
Bj\"orn B\"ottcher for his helpful comments and suggestions.


\end{document}